\documentclass[submission]{FPSAC2021}

\usepackage{comment}
\usepackage{subfigure}
\usepackage{paralist}

\newcommand{\qbinom}[2]{\begin{bmatrix} #1\\#2 \end{bmatrix}}
\newcommand{\txtqbin}[2]{\left[\begin{smallmatrix} #1 \\ #2 \end{smallmatrix} \right]}
\newcommand{\qfactorial}[1]{\left[#1\right]!}
\newcommand{\qnumber}[1]{\left[#1\right]}
\newcommand{\qfalling}[2]{\left[#1\right]_{#2}}

\newcommand{\csf}[1]{X_{#1}({\bf x})}

\newcommand{\csft}[2]{X_{#1}({\bf x}, #2)}

\newcommand{\qhit}[3]{H_{#1}^{#2}(#3)}
\newcommand{\qrook}[2]{R_{#1}(#2)}
\newcommand{\nqhit}[3]{\widetilde{H}_{#1}^{#2}(#3)}

\DeclareMathOperator{\inv}{inv}

\DeclareMathOperator{\asc}{asc}

\DeclareMathOperator{\stat}{stat}

\newtheorem{theorem}{Theorem}[section]
\newtheorem{lemma}[theorem]{Lemma}
\newtheorem{proposition}[theorem]{Proposition}
\newtheorem{corollary}[theorem]{Corollary}
\newtheorem{conjecture}[theorem]{Conjecture}
\newtheorem*{lemma*}{Lemma}
\newtheorem{definition}[theorem]{Definition}
\newtheorem{example}[theorem]{Example}
\newtheorem{remark}[theorem]{Remark}
\newtheorem{notation}[theorem]{Notation}

\definecolor{purple}{rgb}{0.56, 0.0, 1}

\newcommand{\n}{\mathsf{n}}
\newcommand{\e}{\mathsf{e}}

\newcommand{\thmref}[1]{\hyperref[#1]{Theorem~\ref*{#1}}}
\newcommand{\lemref}[1]{\hyperref[#1]{Lemma~\ref*{#1}}}
\newcommand{\propref}[1]{\hyperref[#1]{Proposition~\ref*{#1}}}
\newcommand{\corref}[1]{\hyperref[#1]{Corollary~\ref*{#1}}}
\newcommand{\conjref}[1]{\hyperref[#1]{Conjecture~\ref*{#1}}}
\newcommand{\defref}[1]{\hyperref[#1]{Definition~\ref*{#1}}}
\newcommand{\exref}[1]{\hyperref[#1]{Example~\ref*{#1}}}
\newcommand{\remref}[1]{\hyperref[#1]{Remark~\ref*{#1}}}
\newcommand{\secref}[1]{\hyperref[#1]{Section~\ref*{#1}}}
\newcommand{\figref}[1]{\hyperref[#1]{Figure~\ref*{#1}}}

\title[Chromatic symmetric functions of Dyck paths and $q$-rook theory]{Chromatic symmetric functions \\of Dyck paths  and $q$-rook theory}

\author[L. Colmenarejo, A. H. Morales, G. Panova]{Laura Colmenarejo\thanks{\href{mailto:lcolmenarejo@umass.edu}{lcolmenarejo@umass.edu}. L. Colmenarejo was partially supported by the AMS-Simons Travel Grant.}\addressmark{1}, Alejandro H. Morales\thanks{\href{mailto:amoralesborr@umass.edu }{ahmorales@math.umass.edu}. A. H. Morales was partially supported by the NSF grant DMS-1855536.}\addressmark{1}, \and Greta Panova\thanks{\href{mailto:gpanova@usc.edu}{gpanova@usc.edu}. G. Panova was partially supported by the NSF grant DMS-1939717.}\addressmark{2}}

\address{\addressmark{1}Department of Mathematics and Statistics, University of Massachusetts, Amherst (USA) \\ \addressmark{2}Department of Mathematics, University of Southern California (USA)}

\received{\today}

\abstract{The chromatic symmetric function (CSF) of Dyck paths of Stanley  and its Shareshian--Wachs $q$-analogue have important  connections to  Hessenberg varieties, diagonal harmonics and LLT polynomials. In the case of, so called, abelian Dyck paths they are also curiously related to placements of non-attacking rooks by results of Stanley--Stembridge (1993) and Guay-Paquet (2013). For the $q$-analogue, these  results have been generalized by  Abreu--Nigro (2020) and Guay-Paquet (private communication), using  $q$-hit numbers, which are a variant of the ones introduced by Garsia and Remmel.  Among our main results is a new proof of Guay-Paquet's elegant identity expressing the $q$-CSFs in a CSF basis with $q$-hit coefficients.  We further show its equivalence to the Abreu--Nigro identity expanding the $q$-CSF in the elementary symmetric functions.
}

\keywords{chromatic symmetric functions, abelian Dyck paths, $q$-hit numbers, $q$-rook numbers}

\usepackage[backend=bibtex]{biblatex}
\begin{document}

\maketitle

\section{Introduction}

Let $G$ be a graph with vertices $\{v_1,v_2,\ldots,v_n\}$ that are
totally ordered $v_1<v_2<\cdots <v_n$. In~\cite{St1}, Stanley defined the chromatic symmetric function $\csf{G}$ of  $G$  as 
\begin{equation}
\csf{G} = \sum_{\kappa: V\to \mathbb{P}, \text{ proper}} {\bf
  x}^{\kappa} =  \sum_{\kappa: V\to \mathbb{P}, \text{ proper}} x_1^{\#\kappa^{-1}(1)} x_2^{\#\kappa^{-1}(2)}\cdots, 
\end{equation}
where $\mathbb{P}=\{1,2,3,\ldots\}$, ${\bf x}=(x_1,x_2,\ldots)$, and
the sum is over the proper colorings of the vertices of $G$.

Stanley and Stembridge~\cite{StSt} conjectured that the chromatic symmetric
functions  expand with positive coefficients in the basis $\{e_{\lambda}\}$
of elementary symmetric functions for the following  graphs. Given a Dyck path $d$ from $(0,0)$ to $(n,n)$, let $G(d)$ be the graph with vertices $\{1\ldots n\}$ and edges  $(i,j)$, $i<j$ iff the cell
$(i,j)$ is below the path $d$. These are also the  incomparability graphs of {\em unit interval orders} or graphs obtained from {\em Hessenberg sequences}.

 Shareshian--Wachs~\cite{ShW2} introduced a quasisymmetric version of $\csf{G}$
 defined by 
\begin{align*}
\csft{G}{q} = \sum_{\kappa: V\to \mathbb{P}, \text{ proper}} q^{\asc(\kappa)}{\bf x}^{\kappa},
\end{align*}
where $\asc(\kappa)$ is the number of edges $\{v_i,v_j\}$ of $G$ with
$i<j$ and $\kappa(v_i)<\kappa(v_j)$. 

For the graphs $G(d)$ coming from Dyck paths, the quasisymmetric function $\csft{G(d)}{q}$ is
actually symmetric and Shareshian--Wachs gave a refinement of the
Stanley--Stembridge conjecture for this Catalan family of graphs.

\begin{conjecture}[Stanley--Stembridge,
  Shareshian--Wachs] \label{conj: StaSteShWa}
Let $d$ be a Dyck path then the coefficients of $\csft{G(d)}{q}$ in
the elementary basis are in $\mathbb{N}[q]$.
\end{conjecture}

The symmetric functions $\csft{G(d)}{q}$ are very actively studied thanks to their connections to \emph{Hessenberg varieties}~\cite{ShW2}, \emph{diagonal harmonics}~\cite{CM}, and {\em Macdonald polynomials}~\cite{AP}.

Conjecture~\ref{conj: StaSteShWa} has been verified independently and by different techniques by Cho--Huh~\cite{ChoHuh}, Hamada--Precup~\cite{HP} ,  and Abreu--Nigro~\cite{AN} for the case of {\em abelian} Dyck paths: paths $d$ of size $m+n$ of the form $\n^m w(\lambda) \e^n$ where $w(\lambda)$ is the encoding
in North ($\n$) and East ($\e$) steps of the partition $\lambda \subset n\times m$ (see
Figure~\ref{fig:hitex}). We denote the associated graph by $G(\lambda)$ and the chromatic symmetric function by $\csft{\lambda}{q} = \csft{G(\lambda)}{q}$.

The symmetric functions of abelian Dyck paths  are deeply related to the  {\em $q$-rook theory} of Garsia--Remmel~\cite{GR} as was unveiled in the Abreu--Nigro expansion, itself a $q$-analogue of a result of Stanley--Stembridge \cite{StSt}. The following statements use the standard notation $[n]_k = [n][n-1]\cdots[n-k+1]$, $[n]!=[n]_n$, $\txtqbin{n}{k} =[n]_k/[k]!$,  where $[x]=(1-q^x)/(1-q)$. Also $\qhit{j}{n}{\lambda}$ denotes  $q$-hit numbers~\cite{D} which are equal to the Garsia--Remmel $q$-hit numbers~\cite{GR} up to a power of $q$. Moreover, the $\qhit{j}{n}{\lambda}$ are symmetric  polynomials in $\mathbb{N}[q]$ that at $q=1$ give the number of permutations of size $n$ with permutation matrix having support of size $j$ in the board of$\ \lambda$.

\begin{theorem}[Abreu--Nigro~\cite{AN}]\label{AN:generalLambda}
Let $\lambda$ be partition inside an $n\times m$ board with $\ell(\lambda) = k \leq \lambda_1$. Then
\begin{align*}
\csft{\lambda}{q} &=
\qfactorial{k}\qhit{k}{n+m-k}{\lambda}\cdot e_{m+n-k,k} + \sum_{j=0}^{k-1}
q^j \qfactorial{j}\qnumber{n+m-2j} \qhit{j}{m+n-j-1}{\lambda} \cdot e_{m+n-j,j}.
\end{align*}
\end{theorem}

Central to this paper is a new identity of Guay-Paquet (private communication~\cite{MGP_LR}) that expands $\csft{\lambda}{q}$ in terms of
chromatic symmetric functions for rectangular shape with  coefficients given by {\em $q$-hit numbers} of {\em rectangular boards} of size $n\times m$ that we denote by $\qhit{j}{m,n}{\lambda}$ and satisfy $\sum_{j=0}^n \qhit{j}{m,n}{\lambda}=[m]_n$.    

\begin{theorem}[Guay-Paquet \cite{MGP_LR}] \label{thm:qhitCSFabelian}
Let $\lambda$ be partition inside an $n\times m$ board ($n\leq m$). Then 
\[
\csft{\lambda}{q} =\dfrac{1}{\qfalling{m}{n}} \sum_{j=0}^{n}
\qhit{j}{m,n}{\lambda} \cdot \csft{m^j}{q}.
\]
\end{theorem}

The original proofs of the two statements above  use a linear relation satisfied by $\csft{G(d)}{q}$ called the {\em modular relation}~\cite{AN,AS,MGP}. Our {\bf \em first main result} is an elementary proof of Theorem~\ref{thm:qhitCSFabelian} using a desymmetrizing recursive relation and $q$-rook theory (Section~\ref{sec: pf of MGP}) and {\bf\em our second main result} is the equivalence of this result and Theorem~\ref{AN:generalLambda} (Section~\ref{sec: MGP to AN}). As a by-product of our arguments, we obtain a new proof of  the Abreu--Nigro expansion, a new recurrence to compute $\csft{\lambda}{q}$ (Lemma~\ref{lem:X_recursion}), and new relations of {\em $q$-rook numbers} and {\em $q$-hit numbers} (Lemma~\ref{prop:q-rook identities} and~\ref{prop:qhit-relations}) that develop further the $q$-rook theory of rectangular boards~\cite{LM17}. 

The full version of this paper is available at~\cite{CMP}. 

\section{Background on $q$-rook theory}
For the rest of the paper, we assume $m$ and $n$ are non-negative integers with $m\geq n$. 

\subsection{$q$-rook numbers}

\begin{definition}[$q$-rook numbers \cite{GR}]\label{def: GR q-rook numbers}
Given a partition
$\lambda=(\lambda_1,\lambda_2,\ldots,\lambda_{n})$ inside an $n\times m$ board, the \emph{Garsia-Remmel $q$-rook numbers}
 are defined as $R_k(\lambda) = \sum_p q^{\inv(p)}$,
where the sum is over all placements $p$ of $k$ non-attacking rooks on $\lambda$ and
$\inv(p)$ is the number of cells of $\lambda$ left after each rook cancels its cell, the cells  North in its column and the cells West in its row (see Figure~\ref{fig:ex rook placement}). 
\end{definition}

\begin{proposition}[Garsia-Remmel~\cite{GR}]\label{prop:GR_rook_gen_fun}
Given a partition $\lambda=(\lambda_1,\ldots,\lambda_{\ell})$ we have that 
\begin{equation} \label{eq: def F}
F(x;\lambda):=\sum_{k=0}^{\ell} R_{k}(\lambda)[x]_{\ell-k}  = \prod_{i=1}^{\ell} [x+\lambda_{\ell-i+1}-i+1].
\end{equation}
\end{proposition}

 \begin{lemma}\label{lem:F_ratio}
 Given a partition $\lambda=(\lambda_1,\ldots,\lambda_{\ell})$ we have that 
    $$q^{\lambda_1} [x] \frac{ F(x-1;\lambda)}{F(x;\lambda)} =  [x-\ell+\lambda_1]  -  \sum_{j=1}^{\lambda_1} q^{\lambda_1-j}\prod_{t=1}^{\lambda_j'} \frac{[x+\lambda_{t}-1-\ell+t] }{[x+\lambda_{t}-\ell+t]}.$$
    \end{lemma}    

\begin{proof}
We use induction on $\ell(\lambda)$ and apply Proposition~\ref{prop:GR_rook_gen_fun}. For $\ell(\lambda)=1$, we have 
\begin{multline*}
    [x] \frac{ F(x-1;\lambda)}{F(x;\lambda)} =  [x-1] +  \sum_{j=1}^{\lambda_1} q^{-j}  -  q^{-j}  \frac{[x+\lambda_{1}-1] }{[x+\lambda_{1}]} = q^{-\lambda_1} \frac{[x+\lambda_1-1]}{[x+\lambda_1]}([x+\lambda_1]-[\lambda_1]).
\end{multline*}
Next, expanding the RHS of the above identity and doing standard manipulations gives
  \begin{multline*}
 [x-\ell+\lambda_1]  -  \sum_{j=1}^{\lambda_1} q^{\lambda_1-j}\prod_{t=1}^{\lambda_j'} \frac{[x+\lambda_{t}-1-\ell+t] }{[x+\lambda_{t}-\ell+t]} = \\
  = q^{\lambda_1-\lambda_2} \frac{[x+\lambda_1-\ell]}{[x+\lambda_1-\ell+1]} \left( [x+\lambda_2 -(\ell-1)] - \sum_{j=1}^{\lambda_2} q^{\lambda_2-j} \prod_{t=2}^{\lambda_j'} \frac{[x+\lambda_{t}-1-\ell+t] }{[x+\lambda_{t}-\ell+t]} \right).
\end{multline*}
By induction hypothesis the parenthetical on the RHS above is $q^{\lambda_2} [x]  F(x-1;\tilde{\lambda})/F(x;\tilde{\lambda})$ where $\tilde{\lambda} = (\lambda_2,\ldots,\lambda_\ell)$. Using  $\tilde{\lambda}_t = \lambda_{t+1}$ for the reindexing we obtain the result. 
\end{proof}

\subsection{$q$-hit numbers}
The $q$-hit numbers are defined in terms of the $q$-rook numbers by a change of basis. Let $(a;q)_k=\prod_{i=0}^{k-1}(1-aq^i)$ denote the {\em $q$-Pochhammer symbol}.

\begin{definition}[{\cite[Def. 3.1, Prop. 3.5]{LM17}}]\label{def: hit in terms of rs}
For $\lambda$ inside an $n\times m$ board, we define the \emph{$q$-hit polynomial} of $\lambda$ by 
\begin{equation} \label{eq: hit rook change of basis}
\sum_{i=0}^n \qhit{i}{m,n}{\lambda} x^i := \frac{q^{-|\lambda|}}{\qfactorial{m-n}}\sum_{i=0}^n \qrook{i}{\lambda} \qfactorial{m-i} (-1)^i q^{mi-\binom{i}{2}} (x;q)_i,
\end{equation}
where the coefficients $\qhit{i}{m,n}{\lambda}$ are the \emph{$q$-hit numbers} associated to $\lambda$. 
\end{definition}

\begin{notation}
For square boards with $m=n$, we denote the $q$-hit number by $\qhit{j}{m}{\lambda}$. 
\end{notation}

\begin{remark} \label{rem: difference qhits}
For the case $n=m$, Garsia--Remmel defined $q$-hit numbers $\nqhit{k}{n}{\lambda}$ by the relation
\begin{equation} \label{eq: def GR qhit}
\sum_{i=0}^n \nqhit{i}{n}{\lambda} x^i = \sum_{i=0}^n \qrook{i}{\lambda} [n-i]! \prod_{k=n-i+1}^n (x-q^k).
\end{equation}
One can show that the Garsia--Remmel $q$-hit numbers and our $q$-hit numbers differ by a power of $q$, namely $\widetilde{H}^n_{k}(\lambda)=q^{|\lambda|-kn} H^n_k(\lambda)$ (see \cite[Appendix]{CMP}).
\end{remark}

For the case of square boards, Garsia and Remmel showed that $\widetilde{H}^n_k(\lambda)$ are in $\mathbb{N}[q]$. Later, Dworkin~\cite{D} and Haglund~\cite{H} found different Mahonian statistics on rook placements that realize the polynomials $\widetilde{H}_j(\lambda)$. Guay-Paquet~\cite{MGP_LR} defined the rectangular $q$-hit numbers using a statistic similar to Dworkin's statistic in \cite{D} that  we define next.

\begin{definition}[Statistic for the $q$-hit numbers]\label{def: Dworking stat qhit}
Let $\lambda$ be a partition inside an $n\times m$ board. Given a placements $p$ of $n$ nonattacking rooks on a $n\times m$ board, with exactly $j$ inside $\lambda$, let $stat(p)$ be the number of cells $c$ in the $n\times m$ board such that (i) there is no rook in $c$, 
(ii) there is no rook above $c$ on the same column, and either, 
(iii) if  $c$ is in $\lambda$ then the rook on the same row of $c$ is in $\lambda$ and to the right of $c$ or (iv) if $c$ is not in $\lambda$ then the rook on same row of $c$ is either in $\lambda$ or to the right of $c$. 
\end{definition}

\begin{remark}
Intuitively, Dworkin's statistic $\stat(p)$ is the number of remaining cells in the $n\times m$ board after: wrapping this board on a vertical cylinder and each rook of $p$ cancels the cells South in its column and the cells East in its row until the border of $\lambda$.
\end{remark}

\begin{theorem} \label{thm: qhit statistic rectangular board}
Let $\lambda$ be a partition inside an $n\times m$ board and $j=0,\ldots,n$ then \newline
$\qhit{j}{m,n}{\lambda}= \sum_p q^{stat(p)}$, 
where the sum is over all placements $p$ of $n$ non-attacking rooks on a $n\times m$ board, with exactly $j$ rooks inside $\lambda$.
\end{theorem}

The proof of this result is included in \cite[Appendix]{CMP}. Moreover, for each partition $\lambda$, the statistic $\stat(\cdot)$ is \emph{Mahonian}. This results follows readily from Definition~\ref{def: hit in terms of rs}.

\begin{corollary}
Let $\lambda$ be a partition inside an $n\times m$ board, then
$\displaystyle{\sum_{j=0}^n \qhit{j}{m,n}{\lambda} = [m]_n}$.
\end{corollary}

\begin{example}
Consider the partition $\lambda=(6,3,3,1)$ inside a $6\times 8$ board. In Figure~\ref{fig:ex rook placement}, we present an example of a placement $p$ of two rooks on $\lambda$ with $\inv(p)=7$ and an example of a placement $p'$ of six rooks on the  $6\times 8$ board with two hits on $\lambda$ and $\stat(p')=13$.
\end{example}

\begin{figure}[ht]
\centering
\includegraphics{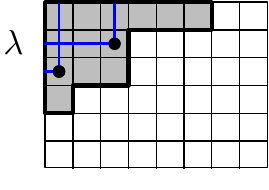}
\quad
\raisebox{5pt}{
\includegraphics{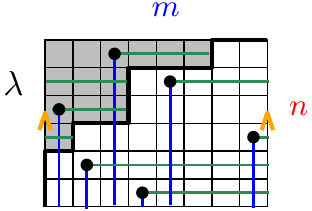}}
\qquad  
\includegraphics{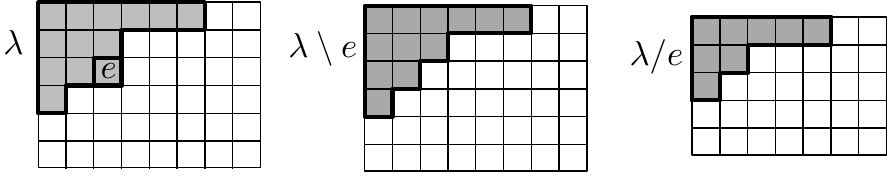}
\caption{Left: Example of the statistics of a $q$-rook number and a $q$-hit number. Right: Example of the deletion and contraction of the board of a partition $\lambda$.}
\label{fig:ex rook placement}
\end{figure}

We finish this section with some results for $q$-hit numbers. 
First, we give a deletion/contraction relation.  Given a shape $\lambda$ and a corner cell $e$ in $\lambda$, $\lambda \backslash e$  denotes the shape obtained after deleting the cell $e$ in $\lambda$, and $\lambda /e$ denotes the shape obtained after deleting in $\lambda$ the row and column containing $e$. See Figure~\ref{fig:ex rook placement} for an example. 
\begin{lemma} \label{lem: deletion/contration}
We have the following deletion/contraction relation:
\begin{equation}\label{lemma:del con hits}
\qhit{j}{m,n}{\lambda}  = \qhit{j}{m,n}{\lambda\backslash e}
+ q^{|\lambda/e|-|\lambda|+j+m-1}\left( \qhit{j-1}{m-1,n-1}{\lambda/e} - q \qhit{j}{m-1,n-1}{\lambda/e}\right).
\end{equation}
\end{lemma}

The next results show the relation between the $q$-hit numbers when we change the dimensions of the board.
\begin{lemma}
Let $\lambda$ be a partition inside an $n\times m$ board. Then,
$\qhit{j}{m,n}{\lambda} = \frac{1}{[m-n]!}\qhit{j}{m,m}{\lambda}$.
\end{lemma}

\begin{lemma} \label{lemma:remove column}
Let $\lambda$ be a partition inside an $(n-1)\times m$ board. Then,
\begin{equation} \label{eq:greta-relation}
    \qhit{j}{m,n}{\lambda} = \qnumber{m+1-n} \qhit{j}{m,n-1}{\lambda}.
\end{equation}
\end{lemma}

Finally, we give formulas for $q$-hit numbers for rectangular shapes and omit their standard proofs. 
\begin{proposition} \label{prop:qhit_rectangles}
$\qhit{k}{N}{m^j}= q^{ (N-j-m+k)k} \qfalling{m}{k} \qfactorial{N-j}  \dfrac{ \qfalling{N-m}{j-k} \qfalling{j}{j-k}}{ \qfactorial{j-k}}.$
\end{proposition}

\begin{proposition} \label{prop: hits small rect in rect}
$\qhit{r}{m,n}{(m-1)^k} = 0$ for $0\leq r\leq k-2$, $\qhit{k-1}{m,n}{(m-1)^k}=[k] [m-1]_{n-1}$, and $\qhit{k}{m,n}{(m-1)^k}=q^k[m-k] [m-1]_{n-1}$.
\end{proposition}

\section{The Guay-Paquet $q$-hit identity} \label{sec: pf of MGP}

In this section we sketch our main result, a proof of Theorem~\ref{thm:qhitCSFabelian} using $q$-rook theory. We start by giving an example of this elegant identity.

\begin{example} \label{ex:tcase}
For $\lambda = (2,1)$ inside a $2\times 3$ board, looking at Figure~\ref{fig:hitex}, we see that
$\qhit{0}{3,2}{\lambda}=q^0=1$, $\qhit{1}{3,2}{\lambda}=2q+2q^2$,
$\qhit{2}{3,2}{\lambda}=q^3$. One can verify that  
\[
\csft{21}{q} = \dfrac{1}{[3][2]}  \left(
  \csft{3^0}{q} + (2q^2+2q)\csft{3^1}{q}+q^3\csft{3^2}{q}\right), 
\]
\end{example}

\begin{figure}
\centering
\subfigure{
\raisebox{25pt}{\includegraphics{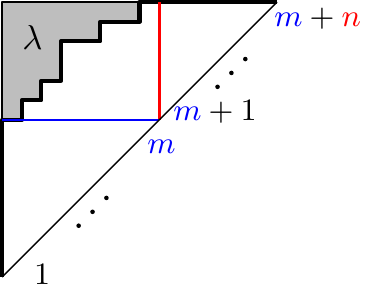}}}
\subfigure{
\includegraphics{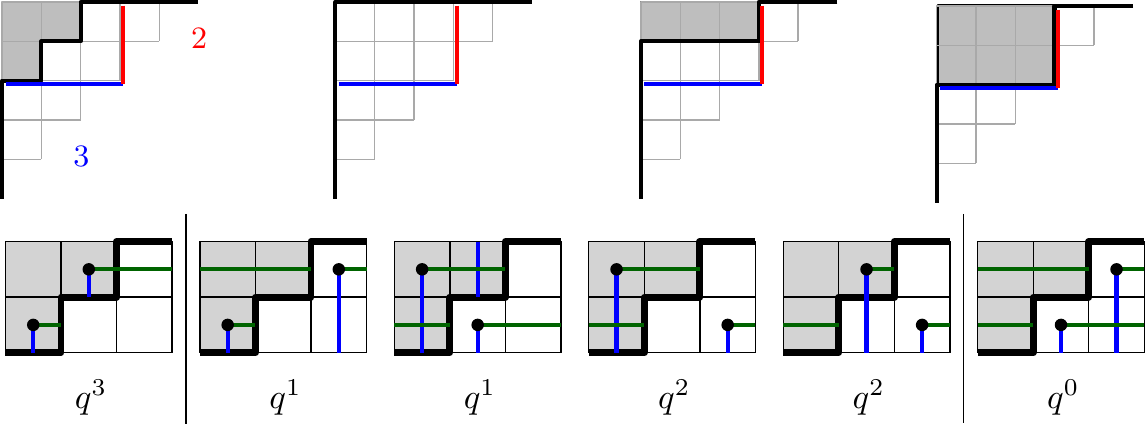}}
\caption{Left: an abelian Dyck path $\lambda$ inside an $n\times m$ board. Top right: the paths for $\lambda=(2,1)$, and for the rectangles $3^0,3^1,3^2$ inside a $2\times 3$ board. Bottom right: the six placements of $2$ rooks in $2\times 3$
  divided by how many rooks ``hit''  $(2,1)$ (in gray) and the
  associated statistic  to each rook placement.} 
\label{fig:hitex}
\end{figure}

We consider the chromatic symmetric functions in variables $x_1,\ldots,x_M$ and each monomial appearing as a particular assignment of the variables (i.e. colors) to the vertices. That is, the vertices $1,\ldots,N=m+n$ are colored $\{1,\ldots,M\}$. For simplicity, we denote by $X_\lambda^N(M)$ the chromatic symmetric polynomial $X_{G(\lambda)}(x_1,\ldots,x_M;q)$ where the graph $G(\lambda)$ has  $N$ vertices. We will use induction on both $M$ and $n,\ m$ when necessary, driven by the following recursion.
\begin{lemma}\label{lem:X_recursion}
For $\lambda \subset m \times n$ we have the following recursion
\begin{align*}
    X_{\lambda}^{m+n}(M) =& X_{\lambda}^{m+n-1}(M-1) 
    + x_M\sum_{i=1}^{m+n} q^{m+n-i - \lambda_i'} X^{m+n-1}_{\lambda/i}(M-1) \\
    &+ x_M^2\sum_{(i,j) \in \lambda} q^{ i-1 + (m+n-j-\lambda_j')} X^{m+n-2}_{\lambda/(i,j)}(M-1),
\end{align*}
where $\lambda/(i,j)$ is the partition obtained by removing row $i$ and column $j$ from $\lambda$, and $\lambda/i$ means we remove from $\lambda$ column $i$, for $i=1,\ldots,m$ or row $m+n-i+1$, for $i=m+1,\ldots,m+n$. 
\end{lemma}
\begin{proof}
In the abelian case, the graph $G(\lambda)$ consists of a clique with vertices $\{1,\ldots m\}$, a clique with vertices $\{m+1,\ldots,m+n\}$ and a bipartite graph in between with edges $(i,m+j)$ for each $(i,j)$ in $\overline{\lambda}$. So a coloring of this graph has at most two vertices of the same color. If the colors used are in $\{1,\ldots,M\}$, there are three cases for the color $M$:
\begin{compactitem}
\item[\textbf{1.}] No vertex is colored $M$, this term contributes  $X_\lambda^{m+n-1}(M-1)$ to $X_\lambda^{m+n}(M)$.
\item[\textbf{2.}] Only one vertex is colored $M$. Suppose this vertex is in column $j$ (from left) and row $i=N-j$ (from top to bottom). It creates ascents with all vertices above it  but not in $\lambda$, giving $N-j-\lambda_j'$ ascents. Deleting this vertex corresponds to deleting its row and column (only one would be a row/column of $\lambda$) and we get a graph on $N-1$ vertices with shape $\lambda/j$ (deleting row $N-j$, column $j$ or row $j$, column $N-j$). These terms contribute 
    $x_M \sum_{j} q^{N-j-\lambda'_j}  X_{\lambda/j}(M-1)$.
\item[\textbf{3.}] Two vertices are colored $M$. Suppose that one is in column $j$ and the other one is in row $i$, necessarily with $(i,j) \in \lambda$. The ascents they contribute are $N-j - \lambda_j'+ i-1$. We can  remove these two vertices, by removing row $i$ and column $j$ from $\lambda$ and decreasing $N$ by 2.
These terms contribute
$x_M^2\sum_{(i,j)} q^{N-j-\lambda'_j + i-1 }  X^{N-2}_{\lambda/(i,j)}(M-1)$.
\end{compactitem}
\end{proof}
For rectangular shapes $\lambda = (m^k)$, Lemma~\ref{lem:X_recursion} gives  the following recursive expansion.
\begin{lemma}\label{lem:X_rect_recursion}
\begin{align*}
 X^{m+n}_{m^k}(M) &= X^{m+n}_{m^k}(M-1) + x_M \left(q^{n-k}[m] X^{m+n-1}_{(m-1)^k}(M-1) + [k]X^{m+n-1}_{m^{k-1}}(M-1) \right. \\  &\left. + q^k[n-k] X^{m+n-1}_{m^k}(M-1) \right)+
x_M^2 q^{n-k}[k][m] X^{m+n-2}_{(m-1)^{k-1}}(M-1). 
\end{align*}
\end{lemma}

\begin{proof}
This follows by carefully applying Lemma~\ref{lem:X_recursion} to the shape $\lambda = m^k$, since $\lambda/i$ is either $(m-1)^k$ or $m^{k-1}$ and $\lambda/(i,j) = (m-1)^{k-1}$.
\end{proof}

\begin{proof}[Proof of Theorem~\ref{thm:qhitCSFabelian}]
Translating Theorem~\ref{thm:qhitCSFabelian} into chromatic symmetric polynomials, we want to prove that for every $M$ we have
\begin{align}\label{eq:GPpolynomials}
X_\lambda(M) =\dfrac{1}{\qfalling{m}{n}} \sum_{j=0}^{n}
\qhit{j}{m,n}{\lambda} \cdot X_{m^j}(M).
\end{align}
We apply Lemma~\ref{lem:X_rect_recursion} to each term $X_{m^j}(M)$ appearing in the right hand side of the formula in~\eqref{eq:GPpolynomials}. We also apply   Lemma~\ref{lem:X_recursion} and the induction hypothesis to the left hand side of the formula in~\eqref{eq:GPpolynomials}, i.e. to $X_{\lambda}(M)$. Then, we obtain an expression where both sides are written in terms of $X_{m^k}^{m+n}(M-1)$, $x_M X^{m+n-1}_{m^k}(M-1)$, $x_M X^{m+n-1}_{(m-1)^k}(M-1)$ and $x_M^2 X^{m+n-2}_{(m-1)^k}(M-1)$. 
The  term at $x_M^0$ corresponds to the coefficient of  $X^{m+n}_{m^k}(M-1)$, which is $\frac{1}{[m]_n} H^{m,n}_k(\lambda)$ on both sides by induction on $M$. For the linear term, matching the coefficients of $x_M X^{m+n-1}_{(m)^k}(M-1)$ and $x_M X^{m+n-1}_{(m-1)^k}(M-1)$ separately, is equivalent to two $q$-hit identities we need to prove:
\begin{align}
\sum_{i=1}^{m}  q^{m+n-i - \lambda_i'} H_k^{m-1,n}(\lambda/(m+n-i,i)) &= [m-n] H_k^{m,n}(\lambda) q^{n-k},\label{Eq:linearterm1}\\
[m-n+1]\sum_{i=m+1}^{m+n}  q^{m+n-i - \lambda_i'} H_k^{m,n-1}(\lambda/(m+n-i,i)) &=  H_k^{m,n}(\lambda)q^k[n-k] + H_{k+1}^{m,n}(\lambda)[k+1]. \label{Eq:linearterm2}
\end{align}
Finally, the quadratic term $x_M^2X^{m+n-2}_{(m-1)^k}(M-1)$ corresponds to the $q$-hit identity:
\begin{align}\label{Eq:quadraticterm}
 q^k\sum_{(i,j)\in \lambda} q^{i+(m-j-\lambda'_j)} H^{m-1,n-1}_k(\lambda/(i,j)) =  [k+1] H^{m,n}_{k+1}(\lambda). 
\end{align}
By Definition~\ref{def: hit in terms of rs}, we translate these three identities involving $q$-hit numbers into three identities involving $q$-rook numbers that are in Lemma~\ref{prop:q-rook identities}. These identities complete the proof of Theorem~\ref{thm:qhitCSFabelian}.
\end{proof}

\begin{lemma}\label{prop:q-rook identities}
The $q$-hit identities~\eqref{Eq:linearterm1}, ~\eqref{Eq:linearterm2}, and~\eqref{Eq:quadraticterm} are equivalent, in that order, to:
\begin{align}
\sum_{j=1}^m q^{m-j} R_k(\lambda/j)  &=    R_k(\lambda) [m-k]    -  R_{k+1}(\lambda)  (q^m -q^{m-k-1}), \label{eq: rook first linear} \\
\sum_{i=1}^n q^{i-1 +\lambda_i} R_k(\lambda/i)  
    &= ([n]-[k])R_k(\lambda), \label{eq: rook second linear}  \\
\sum_{(i,j) \in \lambda} q^{i-j+\lambda_i} R_k(\lambda/(i,j)) &= q[k+1] R_{k+1}(\lambda). \label{eq: rook third linear}
\end{align}
\end{lemma}

We give the proof of the first relation \eqref{eq: rook first linear}. The arguments for the other two relations have a similar flavor.

\begin{proof}[Proof of \eqref{eq: rook first linear}]

Let $\ell = \ell(\lambda)$.
Multiplying on both sides by $[x]_{\ell-i}$ and summing over $i=0,\ldots,\ell$, the claimed relation is equivalent to the generating function identity:
\begin{align*}
    \sum_{j=1}^m q^{m-j} F(x;\lambda/j) &= \sum_i R_i(\lambda) \left( [m-i] [x]_{\ell-i} - (q^m-q^{m-i})[x]_{\ell-i+1} \right) \\
    &= [m+x-\ell]F(x;\lambda) - q^m[x] F(x-1;\lambda),
\end{align*}
where $F(x;\lambda)$ is as in \eqref{eq: def F} and we used the observation that
\[
    [m-i] [x]_{\ell-i} - (q^m-q^{m-i})[x]_{\ell-i+1} 
    = [x]_{\ell-i} [x+m-\ell] - q^m[x]_{\ell-i+1}.
\]

We have that
\[
    F(x;\lambda/j) = \prod_{i=1}^{\lambda'_j} [x+\lambda_i -1 -\ell+i ] \prod_{i=\lambda'_j+1}^\ell [x+\lambda_i -\ell+i] = F(x;\lambda) \prod_{i=1}^{\lambda'_j} \frac{ [x-1 +\lambda_i -\ell+i ]}{[x+\lambda_i -\ell+i ]}.
\]

Using Lemma~\ref{lem:F_ratio}, we have that since $\lambda/j = \lambda$ if $j>\lambda_1$,
\begin{align*}
    \sum_{j=1}^m q^{m-j} F(x;\lambda/j) &= F(x;\lambda)\bigg( [m-\lambda_1]  + q^{m-\lambda_1} \sum_{j=1}^m q^{\lambda_1 - j}\prod_{i=1}^{\lambda'_j} \frac{ [x-1 +\lambda_i -\ell+i ]}{[x+\lambda_i -\ell+i ]}\bigg) \\
    &=  F(x;\lambda)\left( [m-\lambda_1] +q^{m-\lambda_1}[x-\ell+\lambda_1] -q^{m-\lambda_1}q^{\lambda_1}[x] \frac{F(x-1;\lambda)}{F(x;\lambda) }\right) \\
    &= F(x;\lambda)[x-\ell+\lambda_1 + m-\lambda_1] -q ^m[x]F(x-1;\lambda),
\end{align*}
as desired.
\end{proof}

\section{The Abreu--Nigro expansion in the elementary basis} \label{sec: MGP to AN}
 
In this section we show that Guay-Paquet's identity (Theorem~\ref{thm:qhitCSFabelian}) is equivalent to Abreu--Nigro's identity presented (Theorem~\ref{AN:generalLambda}). 
We start by giving a proof of Abreu--Nigro's identity for rectangular shapes. 
\begin{lemma}[Abreu--Nigro's formula for rectangles]\label{lem:ANrectangle}
\begin{align*}
    \csft{m^k}{q} &=\qfactorial{k}\qhit{k}{m+n-k}{m^k}\cdot e_{m+n-k,k} 
    + \sum_{r=0}^{k-1}q^r \qfactorial{r}\qnumber{n+m-2r} \qhit{r}{m+n-r-1}{m^k} \cdot e_{m+n-r,r}.
\end{align*}
\end{lemma}

In order to prove this case of the Abreu--Nigro identity we need the following result.

\begin{lemma}[Guay-Paquet formula for rectangles] \label{lem: MGP for rectangles}
\begin{equation} \label{eq: MGP for rectangles}
[m]X_{(m-1)^k}^{m+n-1} =q^k[m-k]X_{m^k} + [k]X_{m^{k-1}}.
\end{equation}
\end{lemma}

\begin{proof}
By Theorem~\ref{thm:qhitCSFabelian} for the shape $\lambda=(m-1)^k \subset n\times m$ and the formula for the $q$-hit numbers $H^{m,n}_r((m-1)^k)$ from Proposition~\ref{prop: hits small rect in rect} we obtain
\begin{align*}
X_{(m-1)^k} &= \frac{1}{[m]_{n}} q^k[m-1]_k [m-k]_{n-k} X_{m^k} + \frac{1}{[m]_{n}}[k][m-1]_{k-1}[m-k]_{n-k} X_{m^{k-1}},\\
[m]X_{(m-1)^k} &= q^k [m-k] X_{m^k} + [k]X_{m^{k-1}}.
\end{align*}
\end{proof}

\begin{proof}[Proof sketch of Lemma~\ref{lem:ANrectangle}]
We use induction on $m$ and $k$. For the base case, note that $X_{m^0} = [m+n]! e_{m+n} = H_0^{m+n}(m^0) e_{m+n}$. By Lemma~\ref{lem: MGP for rectangles} we have that 
\[
X_{m^k} = \frac{1}{q^k[m-k]}\left([m] X_{(m-1)^k} - [k]X_{m^{k-1}}\right).
\]
Next, we use the induction hypothesis on $X_{(m-1)^k}$ and $X_{m^{k-1}}$,  Proposition~\ref{prop:qhit_rectangles} for the hit numbers of rectangles, and routine simplifications to verify  the desired formula for $X_{m^k}$.
\end{proof}

We are now ready to prove that the Guay-Paquet's identity and Abreu--Nigro's follow from each other. As a corollary, we obtain a new proof of the latter.

\begin{proof}[Proof of Theorem~\ref{AN:generalLambda}]

Applying Lemma~\ref{lem:ANrectangle} to the RHS of the formula in Theorem~\ref{thm:qhitCSFabelian}, we obtain that
\begin{align*}
    &\dfrac{1}{\qfalling{m}{n}} \sum_{j=0}^{n}
\qhit{j}{m,n}{\lambda} \cdot \csft{m^j}{q} 
= \dfrac{1}{\qfalling{m}{n}} \sum_{j=0}^{n}
\qhit{j}{m,n}{\lambda} \left( \qfactorial{j}\qhit{j}{n+m-j}{m^j}\cdot e_{m+n-j,j}\right) \\[0.07in]
&+\dfrac{1}{\qfalling{m}{n}} \sum_{j=0}^{n}
\qhit{j}{m,n}{\lambda} \left(  q^r\sum_{r=0}^{j-1}
\qfactorial{r} \qnumber{n+m-2r}\qhit{r}{m+n-r-1}{m^j} \cdot e_{m+n-r,r}\right).
\end{align*}
Now, switching the summation order, we have that
\begin{align*}
    &\dfrac{1}{\qfalling{m}{n}} \sum_{j=0}^{n}
\qhit{j}{m,n}{\lambda} \cdot \csft{m^j}{q}
=\sum_{r=0}^n e_{m+n-r,r} \dfrac{1}{\qfalling{m}{n}}\qfactorial{r} \qhit{r}{m+n-r}{m^r}\qhit{r}{m,n}{\lambda} \\
&+ \sum_{r=0}^{n-1} e_{m+n-r,r} \dfrac{1}{\qfalling{m}{n}}\left(q^r\sum_{j=r+1}^n \qfactorial{r} \qnumber{n+m-2r}\qhit{r}{m+n-r-1}{m^j}\qhit{j}{m,n}{\lambda}\right).
\end{align*}
Thus, we need to show that for $r=k=\ell(\lambda)$,
\begin{align*}
    \qfalling{m}{n}  \qhit{k}{m+n-k}{\lambda} &= \qhit{k}{m+n-k}{m^k}\qhit{k}{m,n}{\lambda}
    + q^k\sum_{j=k+1}^n  \qnumber{n+m-2k} \qhit{k}{m+n-k-1}{m^j}\qhit{j}{m,n}{\lambda},
\end{align*}
and for $r<k=\ell(\lambda)$,
\begin{align*}
    \qfalling{m}{n} q^r \qnumber{n+m-2r}\qhit{r}{m+n-r-1}{\lambda} 
    &= \qhit{r}{m+n-r}{m^r}\qhit{r}{m,n}{\lambda} \\
    &+ q^r\sum_{j=r+1}^n   \qnumber{n+m-2r}\qhit{r}{m+n-r-1}{m^j}\qhit{j}{m,n}{\lambda}.\nonumber 
\end{align*}
After using Proposition~\ref{prop:qhit_rectangles}, these two relations are equivalent to the following identities relating $q$-hit numbers of $\lambda$ in square boards and rectangular boards. Finally, the Abreu-Nigro expansion for $\csft{\lambda}{q}$ follows now from Lemma~\ref{prop:qhit-relations}. \end{proof}

\begin{lemma}\label{prop:qhit-relations}
Let $\lambda$ be a partition inside an $n\times m$ board and $k=\ell(\lambda)$, then 
\begin{align}\label{eq:keyrel1}
   \qbinom{m-k}{n-k}\qhit{k}{m+n-k}{\lambda} &= q^{k(n-k)}  \qfalling{m+n-2k}{m-k} \qhit{k}{m,n}{\lambda}. 
\intertext{For $0\leq r< k$, we have}
\label{eq:keyrel2}
   \qbinom{m-r}{n-r}  \qhit{r}{m+n-r-1}{\lambda} &= 
    q^{ r(n-r-1)} \qfalling{m+n-2r-1}{m-r-1} \qhit{r}{m,n}{\lambda} \nonumber\\
    & +  \sum_{j=r+1}^n   q^{r(n-1-j)} \qbinom{j}{r} \dfrac{ \qfalling{m+n-r-j-1}{m-r}  }{\qnumber{n-r}} \qhit{j}{m,n}{\lambda}.
\end{align}

\end{lemma}

\begin{proof}[Proof sketch of Lemma~\ref{prop:qhit-relations}]
Each identity follows by using Definition~\ref{def: hit in terms of rs} to rewrite both the LHS and RHS in terms of $q$-rook numbers $R_j(\lambda)$ and showing the resulting expressions are equal via routine $q$-factorial manipulations. 

We remark that there is a more interesting proof using the deletion/contraction formula in Lemma~\ref{lem: deletion/contration} on both sides of each relation, the rectangle-resizing identity in Lemma~\ref{lemma:remove column}, and induction.
\end{proof}

\section{Open problems}

Since our proof of Theorem~\ref{thm:qhitCSFabelian} uses $q$-rook theory, it would be interesting to find a bijective proof of this result relating colorings with rook placements. 

There are other rules for the elementary basis expansion of $\csft{\lambda}{q}$. In particular, Cho--Huh \cite{ChoHuh} give an expansion in terms of {\em $P$-tableaux} of shape  $2^j 1^{m+n-2j}$ such that there is no $s \geq j+2$ such that  $(a_{i,1},a_{s,1}) \in \lambda$ for all $i \in \{ \ell+1,\ldots,s-1\}$. Let $c_j^{m,n}(q):=\sum_T q^{\inv_{G(\lambda)}(T)}$, where the sum is over such tableaux (see~\cite[Sec. 6]{ShW2}).  It would be interesting to find a weight-preserving bijection that shows that 
\[
c_j^{m,n}(q) = \begin{cases}
[j]! H_j^{m+n-j}(\lambda) &\text{ if } j=\ell(\lambda),\\
q^j[j]! [m+n-2j]H_j^{m+n-j-1}(\lambda) &\text{ if } j<\ell(\lambda).
\end{cases}
\]

\acknowledgements{We thank Mathieu Guay-Paquet for generously sharing the notes~\cite{MGP_LR} with 
Theorem~\ref{thm:qhitCSFabelian} as well as Alex Abreu, Per Alexandersson, and Antonio Nigro for insightful discussions.}

\printbibliography

\end{document}